%% file: Chase.tex
\documentclass[a4paper,12pt]{article}

\usepackage{amscd}
\usepackage{amsfonts}
\usepackage{amsmath}
\usepackage{amssymb}
\usepackage{amsthm}
\usepackage[T1]{fontenc} 
\usepackage{here} 
\usepackage{mathrsfs} 
\usepackage{txfonts} 
\usepackage[all]{xy} 
\usepackage{algorithm} 
\usepackage{algpseudocode} 
\usepackage{diagbox} 

\allowdisplaybreaks

\theoremstyle{plain}
\newtheorem{thm}{Theorem}[section]
\newtheorem{lmm}[thm]{Lemma}

\newtheorem{crl}[thm]{Corollary}

\theoremstyle{definition}
\newtheorem{dfn}[thm]{Definition}

\newcommand{\vs}[1][0.2]{\vspace{#1in}\noindent\ignorespaces}
\newcommand{\ba}{\begin{array*}}
\newcommand{\ea}{\end{array*}}
\newcommand{\be}{\begin{eqnarray*}}
\newcommand{\ee}{\end{eqnarray*}}
\newcommand{\bi}{\begin{itemize}}
\newcommand{\ei}{\end{itemize}}
\newcommand{\bb}{\vs\begin{itembox}}
\newcommand{\eb}{\end{itembox}}
\newcommand{\bc}{\begin{center}}
\newcommand{\ec}{\end{center}}
\newcommand{\bs}{\vs\begin{screen}}
\newcommand{\es}{\end{screen}}

\def\ens#1{{\mathchoice{\left\{ #1 \right\}}{\{ #1 \}}{\{ #1 \}}{\{ #1 \}}}}
\def\set#1#2{{\mathchoice{\left\{ #1 \ \middle| \ #2 \right\}}{\{ #1 \mid #2 \}}{\{ #1 \mid #2 \}}{\{ #1 \mid #2 \}}}}
\def\r#1{\text{\rm #1}}
\def\t#1{\text{#1}}
\def\Bigv#1{\left| #1 \right|}
\def\v#1{{\mathchoice{\Bigv{#1}}{| #1 |}{| #1 |}{| #1 |}}}
\def\Bign#1{\left\| #1 \right\|}
\def\n#1{{\mathchoice{\Bign{#1}}{\| #1 \|}{\| #1 \|}{\| #1 \|}}}

\newcommand{\bC}{\mathbb{C}}

\newcommand{\bN}{\mathbb{N}}

\newcommand{\bR}{\mathbb{R}}

\newcommand{\bZ}{\mathbb{Z}}

\newcommand{\cC}{\mathscr{C}}
\newcommand{\cD}{\mathscr{D}}

\newcommand{\cF}{\mathscr{F}}

\newcommand{\cI}{\mathscr{I}}

\newcommand{\cM}{\mathscr{M}}

\newcommand{\cP}{\mathscr{P}}

\newcommand{\rC}{\r{C}}

\newcommand{\C}{\bC}

\newcommand{\N}{\bN}

\newcommand{\R}{\bR}
\newcommand{\Z}{\bZ}

\newcommand{\ch}{\r{ch}}

\newcommand{\Hom}{\r{Hom}}

\newcommand{\BAb}{\r{BAb}}

\algnewcommand\algorithmicbreak{{\bf break}}
\algnewcommand\Break{\algorithmicbreak{}}
\algnewcommand\algorithmiccontinue{{\bf continue}}
\algnewcommand\Continue{\algorithmiccontinue{}}

\newcommand{\fn}[2]{
  \footnote[0]{
    $
    \begin{array}{l}
      \r{MSC2020: #1} \\
      \r{Key words: #2}
    \end{array}
    $
  }
}

\title{Non-Archimedean Analogue of Chase's Lemma}
\author{Tomoki Mihara}
\date{}

\begin{document}

\maketitle
\input{Abstract}
\tableofcontents
\fn{11U07, 03E55, 20K25}{non-Archimedean analysis, Chase's lemma, measurable cardinal}

\input{Introduction}
\input{Convention}
\input{Chases_Lemma}

\input{Edas_Theorem}

\input{References}

\end{document}

%% file: Abstract.tex
\begin{abstract}
We formulate and verify a non-Archimedean analogue of Chase's lemma. Following the framework by K.\ Eda removing restriction of cardinality from analogy on direct product between countability and non-$\omega_1$-measurability, we extend the non-Archimedean analogue of Chase's lemma to a non-Archimedean counterpart of the extension by K.\ Eda of the extension by M.\ Dugas and B.\ Zimmermann-Huisgen of Chase's lemma.
\end{abstract}

%% file: Introduction.tex
\section{Introduction}
\label{Introduction}

Let $R$ be a ring. For an $R$-module $M$, we denote by $M^{\vee}$ the $R$-linear dual $\Hom_R(M,R)$. For any family $(M_i)_{i \in I}$ of $R$-modules, the dual functor sends the direct sum $\bigoplus_{i \in I} M_i$ to the direct product $\prod_{i \in I} (M_i^{\vee})$, but does not necessarily send the direct product $\prod_{i \in I} M_i$ to the direct sum $\bigoplus_{i \in I} (M_i^{\vee})$. In this sense, the duality between direct sum and direct product is one-sided.

\vs
The converse direction is closely related to the notion of reflexivity. We say that an $R$-module $M$ is {\it reflexive} if the canonical morphism $M \to M^{\vee \vee}$ is an isomorphism. If a class $\cC$ of reflexive $R$-modules is closed under direct sum with an index set $I$, then we have
\be
\left( \prod_{i \in I} M_i \right)^{\vee} \cong \left( \prod_{i \in I} (M_i^{\vee \vee}) \right)^{\vee} \cong \left( \bigoplus_{i \in I} (M_i^{\vee}) \right)^{\vee \vee} \cong \bigoplus_{i \in I} (M_i^{\vee})
\ee
for any family $(M_i)_{i \in I}$ in $\cC$ indexed by $I$.

\vs
For example, the class of finite free $R$-modules satisfies this property for any finite set $I$. However, an infinitely generated $R$-module is rarely reflexive. In particular, if $R$ is a field and $I$ is a set with $\max \ens{\# R,\# \N} \leq \# I$, then we have
\be
\dim_R R^{\oplus I} & = & \# I \\
\dim_R (R^{\oplus I})^{\vee} & = & \dim_R R^I = 2^{\# I} \\
\dim_R (R^{\oplus I})^{\vee \vee} & = & \dim_R (R^{\oplus \cP(I)})^{\vee} = \dim_R R^{\cP(I)} = 2^{2^{\# I}} > \# I,
\ee
where $\cP(I)$ denotes the power set of $I$. On the other hand, if we consider the Archimedean normed setting, the situation drastically changes because of appearance of Hilbert spaces over $\C$: The $\ell^2$-space $\ell^2(I,\C)$ on $I$ over $\C$ is reflexive with respect to the continuous dual for any set $I$. This phenomenon is not common in Banach spaces over $\C$. For example, for any set $I$, the continuous dual of $\ell^1$-space $\ell^1(I,\C)$ of $I$ over $\C$ is the $\ell^{\infty}$-space $\ell^{\infty}(I,\C)$ of $I$ over $\C$, but the continuous dual of $\ell^{\infty}(I,\C) \cong \rC(\beta I,\C)$ includes the evaluation at a non-principal ultrafilter of $I$ if $\# I \geq \# \N$, which cannot correspond to a point in the image of $\ell^1(I,\C)$. Deeper studies are found in the theory of locally convex spaces.

\vs
Going back to the algebraic setting, it is natural to study reflexive modules other than finite free modules. Specker gave the following surprising breakthrough in \cite{Spe50}:

\begin{thm}[Specker's theorem]
For any countable set $I$, the canonical morphism $\Z^{\oplus I} \to (\Z^I)^{\vee}$ is an isomorphism, and hence $\Z^{\oplus I}$ and $\Z^I$ are reflexive.
\end{thm}

{\it Specker phenomenon} refers to this pathological property of $\Z^I$ for a countable set $I$, which is introduced in \cite{Eda83-2} and is named by A.\ Blass in \cite{Bla92}, and has been deeply studied in various settings especially with slender groups and Fuchs-44-groups. In particular, Specker's theorem is extended to what is nowadays called {\it {\L}o\'s's theorem} by A.\ Ehrenfeucht and J.\ {\L}o\'s in \cite{EL54} and by E.\ C.\ Zeeman \cite{Zee55} independently:

\begin{thm}[{\L}o\'s's theorem]
For any set $I$, if $\# I$ is not $\omega_1$-measurable, then the canonical morphism $\Z^{\oplus I} \to (\Z^I)^{\vee}$ is an isomorphism, and hence $\Z^{\oplus I}$ and $\Z^I$ are reflexive.
\end{thm}

Here, the $\omega_1$-measurability, which was shortly mentioned as ``measurability'' in an old convention in a way different from the modern terminology on measurability, means the existence of an $\omega_1$-complete non-principal ultrafilter. The conclusion is equivalent to the factorisation property that every group homomorphism $\Z^I \to \Z$ factors through the quotient map $\Z^I \twoheadrightarrow \Z^I/{\sim}_{I_0}$ for a finite subset $I_0 \subset I$, where ${\sim}_{I_0}$ denotes the equivalence relation given by
\be
(x_i)_{i \in I} \sim_{I_0} (x'_i)_{i \in I} \Leftrightarrow \forall i_0 \in I_0[x_{i_0} = x'_{i_0}].
\ee
K.\ Eda further extended {\L}o\'s's theorem in \cite{Eda82} Corollary 2 in terms of a factorisation property:

\begin{thm}[Eda's theorem]
For any set $I$, every group homomorphism $\Z^I \to \Z$ factors through the quotient map $\Z^I \twoheadrightarrow \Z^I/{\sim}_{U_0}$ for a finite subset $U_0 \subset \beta_{\omega} I$, where $\beta_{\omega} I$ denotes the set of $\omega_1$-complete ultrafilters of $I$ and ${\sim}_{U_0}$ denotes the equivalence relation given by
\be
(x_i)_{i \in I} \sim_{U_0} (x'_i)_{i \in I} \Leftrightarrow \forall \cF \in U_0[\set{i \in I}{x_i = x'_i} \in \cF].
\ee
\end{thm}

The combination of {\L}o\'s's theorem and Eda's theorem is called {\it {\L}o\'s--Eda theorem}. Eda's framework, which removes restriction of cardinality from analogy on direct product between countability and non-$\omega_1$-measurability, is not limited to Specker phenomenon for Abelian groups. Indeed, K.\ Eda applied this frame work to Specker phenomenon in a non-Archimedean normed setting.

\vs
We recall duality theory in non-Archimedean analysis. Let $k$ be a complete valuation field with a non-trivial valuation. We denote by $O_k$ the valuation ring of $k$. W.\ H.\ Schikhof introduced a contravariant categorical equivalence between Banach $k$-vector spaces and an abstract notion called ``embeddable absolutely convex complete edged compactoids'' in \cite{Sch95} Theorem 4.6. Although the equivalence itself is abstractly given by the invertibility of a fully faithful essentially surjective functor, its restriction to the case where $k$ is a local field, which we call {\it Schikhof duality}, can be explicitly described by the continuous dual functors between Banach $k$-vector spaces and compact Hausdorff flat linear topological $O_k$-modules (cf.\ \cite{ST02} Theorem 1.2 and \cite{Mih21-1} Proposition 1.7). Although there are many studies (cf.\ \cite{Sch84} and \cite{Sch02}) on reflexivity of locally convex spaces over $k$ analogous to that of locally convex spaces over $\C$, we concentrate on preceding works closely related to Schikhof duality.

\vs
P.\ Schneider and J.\ Teitelbaum extended Schikhof duality to a duality between Banach $k$-linear representations of a profinite group $G$ and compact Hausdorff flat linear topological $O_k[[G]]$-modules for the case where $k$ is a local field with $\ch(k) = 0$, which we call {\it Schneider--Teitelbaum duality}, in \cite{ST02} Theorem 2.3. We extended Schikhof duality to dualities of several classes of locally convex spaces in \cite{Mih21-1} Theorem 3.6, Theorem 3.20, and Theorem 3.32, and to a duality of specific symmetric monoidal categories in \cite{Mih21-1} Theorem 2.2 to obtain a duality between Abelian groups and rigid analytic Abelian groups as a non-Archimedean analogue of Pontryagin duality in \cite{Mih21-1} Theorem 3.5 and Theorem 3.16. We also extended Schneider--Teitelbaum duality to unitarisable Banach $k$-linear representations of a locally profinite group for the case where $k$ is a local field without the restriction of $\ch(k)$ in \cite{Mih21-2} Theorem 3.17.

\vs
As an analogue of the non-reflexivity of $\ell^1(I,\C)$ and $\ell^1(I,\C)^{\vee} \cong \ell^{\infty}(I,\C)$ for an infinite set $I$, the completion $\rC_0(I,k)$ of $k^{\oplus I}$ with respect to the supremum norm and the $\ell^{\infty}$-space $\rC_0(I,k)^{\vee} \cong \ell^{\infty}(I,k)$ of $I$ over $k$ are not reflexive if $k$ is spherically complete. More generally, as an analogue of the non-reflexivity of infinite dimensional vector spaces, when $k$ is spherically complete, a Banach $k$-vector space $V$ is reflexive if and only if $\dim_k V < \infty$ (cf.\ \cite{Roo78} 4.16).

\vs
However, if $k$ is not spherically complete, then the situation drastically changes: Both of $\rC_0(I,k)$ and $\ell^{\infty}(I,k)$ are reflexive for any countable set $I$, and more generally, every Banach $k$-vector space of countable type and its dual Banach $k$-vector space are reflexive (cf.\ \cite{Roo78} 4.16 and 4.17). The result is extended to the case where $\# I$ is not $\omega_1$-measurable (cf.\ \cite{Roo78} 4.21 and \cite{MN89} \S 12 Corollary 7.18). K.\ Eda further extended it to the case without restriction of $\# I$ (cf.\ \cite{MN89} \S 12 Theorem 7.17).

\vs
K.\ Eda further applied this framework to Chase's lemma, which was originally stated as a theorem (cf.\ \cite{Cha62} Theorem 1.2). Chase's lemma is a useful tool to analyse a homomorphism from a direct product to a direct sum. It roughly states that such a homomorphism essentially vanishes if we ignore finite components of the domain and the codomain and the divisible part of the codomain. Since the precise statement is a little complicated because of the use of a filter of right principal ideals of a non-commutative ring, we instead introduce a specialisation to $\Z$:

\begin{thm}[Chase's lemma]
\label{Chase's lemma}
Let $I$ be a countable set, $J$ a set, $(M_i)_{i \in I}$ a family of Abelian groups indexed by $I$, $(N_j)_{j \in J}$ a family of Abelian groups indexed by $J$, and $f$ a group homomorphism $\prod_{i \in I} M_i \to \bigoplus_{j \in J} N_j$. Then there exists a tuple $(m,I_0,J_0)$ of a positive integer $m$, a finite subset $I_0 \subset I$, and a finite subset $J_0 \subset J$ such that
\be
f \left( \ens{0}^{I_0} \times m \prod_{i \in I \setminus I_0} M_i \right) \subset \bigoplus_{j_0 \in J_0} N_{j_0} \oplus \bigcap_{n \in \N} n \bigoplus_{j \in J \setminus J_0} N_j.
\ee
\end{thm}

Theorem \ref{Chase's lemma} has various generalisations. For example, M.\ Dugas and B.\ Zimmermann-Huisgen extended Theorem \ref{Chase's lemma} to the non-$\omega_1$-measurable setting in \cite{DZH82} Theorem 2:

\begin{thm}[Dugas--Zimmermann-Huisgen's extension of Theorem \ref{Chase's lemma}]
\label{Dugas--Zimmermann-Huisgen's theorem}
Let $I$ be a set, $J$ a set, $(M_i)_{i \in I}$ a family of Abelian groups indexed by $I$, $(N_j)_{j \in J}$ a family of Abelian groups indexed by $J$, and $f$ a group homomorphism $\prod_{i \in I} M_i \to \bigoplus_{j \in J} N_j$. If $\# I$ is not $\omega_1$-measurable, then there exists a tuple $(m,I_0,J_0)$ of a positive integer $m$, a finite subset $I_0 \subset I$, and a finite subset $J_0 \subset J$ such that
\be
f \left( \ens{0}^{I_0} \times m \prod_{i \in I \setminus I_0} M_i \right) \subset \bigoplus_{j_0 \in J_0} N_{j_0} \oplus \bigcap_{n \in \N} n \bigoplus_{j \in J \setminus J_0} N_j.
\ee
\end{thm}

We recall that Specker's theorem and {\L}o\'s's theorem (and also their non-Archimedean non-spherically complete counterparts) are interpreted into statements on a quotient map. Similarly, Theorem \ref{Chase's lemma} and Theorem \ref{Dugas--Zimmermann-Huisgen's theorem} are interpreted into statements on restrictions to the kernel of the quotient map $\prod_{i \in I} M_i \twoheadrightarrow (\prod_{i \in I} M_i)/{\sim}_{I_0}$, where ${\sim}_{I_0}$ is the equivalence relation defined in the same way above. Following the interpretation, K.\ Eda further removed from Theorem \ref{Dugas--Zimmermann-Huisgen's theorem} the restriction of cardinality in \cite{Eda83-1} Theorem 2:

\begin{thm}[Eda's extension of Theorem \ref{Dugas--Zimmermann-Huisgen's theorem}]
\label{Eda's theorem}
Let $I$ be a set, $J$ a set, $(M_i)_{i \in I}$ a family of Abelian groups indexed by $I$, $(N_j)_{j \in J}$ a family of Abelian groups indexed by $J$, and $f$ a group homomorphism $\prod_{i \in I} M_i \to \bigoplus_{j \in J} N_j$. There exists a tuple $(m,U_0,J_0)$ of a positive integer $m$, a finite subset $U_0 \subset \beta_{\omega} I$, and a finite subset $J_0 \subset J$ such that
\be
f \left( m K_{H_0} \right) \subset \bigoplus_{j_0 \in J_0} N_{j_0} \oplus \bigcap_{n \in \N} n \bigoplus_{j \in J \setminus J_0} N_j,
\ee
where $K_{H_0}$ denotes the kernel of the quotient map $\prod_{i \in I} M_i \twoheadrightarrow (\prod_{i \in I} M_i)/{\sim}_{H_0}$.
\end{thm}

K.\ Eda indicated in personal communication the expectation of the existence of a non-Archimedean counterpart of Theorem \ref{Eda's theorem}, following the philosophy of Eda's framework. The aim of this paper is to formulate and verify such a non-Archimedean analogue of Theorem \ref{Eda's theorem}.

\vs
We briefly explain contents of this paper. In \S \ref{Convention}, we introduce convention for this paper. In \S \ref{Chase's Lemma}, we formulate and verify a non-Archimedean analogue of Theorem \ref{Chase's lemma}. In \S \ref{Chase--Dugas--Zimmermann-Huisgen--Eda Theorem}, we formulate and verify a non-Archimedean analogue of Theorem \ref{Eda's theorem}, as an extension of the non-Archimedean analogue of Theorem \ref{Chase's lemma}. In order not to restrict the potential reader, we elaborately recall basic arguments on ideals appearing in studies of Chase's lemma, which are well-known to experts.

%% file: Convention.tex
\section{Convention}
\label{Convention}

We denote by $\omega$ the least transfinite ordinal $\aleph_0$, which is identical to the set of non-negative integers. For a set $X$, we denote by $\# X$ its cardinality, and by $\cP(X)$ the set of subsets of $X$. For a set $X$, an ordinal $\alpha$, and a binary relation $R$ on the class of ordinals, we set $\cP_{R \alpha}(X) \coloneqq \set{U \in \cP(X)}{\# U R \alpha}$.

\vs
For a class $X$ and a set $Y$, we denote by $X^Y$ the class of maps $Y \to X$. When we handle a sequence or a family $s$ indexed by a set $I$, we frequently use the map notation $s(i)$ instead of the subscript notation $s_i$ to point the entry at $i \in I$, in order to avoid massive use of subscripts. For a set $I$ and a family $X$ of sets indexed by $I$, a {\it choice function of $X$} is a map $x \colon I \to \bigsqcup_{i \in I} X(i)$ such that for any $i \in I$, the relation $x(i) \in X(i)$ holds. For a map $f$ and a subset $X'$ of its domain, we denote by $f \upharpoonright X'$ the restriction of $f$ to $X'$, and by $f[X']$ the image of $X'$ by $f$. For a set $I$ and a map $f \colon I \to \R_{\geq 0}$, we denote by $\sup_{i \in I} f(i)$ the supremum of $f[I]$ in $\R_{\geq 0} \sqcup \ens{\infty}$. In particular, $\sup_{i \in I} f(i)$ for the case $I = \emptyset$ is $0$ rather than $- \infty$ in this context.

\vs
For a set $X$, an $x \in X$, and a binary relation $R$ on $X$, we set $X_{R x} \coloneqq \set{x' \in X}{x' R x}$. We note that every $d \in \omega$ is identical to $\omega_{< d}$, and hence for a set $X$, $X^d$ formally means $X^{\omega_{< d}}$, which is naturally identified with the set of $d$-tuples in $X$.

\vs
A {\it (non-Archimedean) normed Abelian group} means an Abelian group $V$ equipped with a map $\n{\cdot} \colon V \to \R_{\geq 0}$ satisfying the following:
\bi
\item[(1)] For any $(v_0,v_1) \in V^2$, the inequality $\n{v_0 - v_1} \leq \max \ens{\n{v_0},\n{v_1}}$ holds.
\item[(2)] For any $v \in V$, the equality $\n{v} = 0$ holds if and only if $v = 0$.
\ei
A {\it (non-Archimedean) Banach Abelian group} is a normed Abelian group such that the ultrametric on $V$ defined by
\be
(v_0,v_1) \mapsto \n{v_0 - v_1}
\ee
is complete. Every closed subgroup $W$ of a Banach Abelian group $V$ forms a Banach Abelian group with respect to the restriction of the structure of $V$, and we always regard $W$ as a Banach Abelian group in this way. We denote by $\BAb$ the class of Banach Abelian groups.

\vs
For Banach Abelian groups $V$ and $W$, a group homomorphism $f \colon V \to W$ is said to be {\it bounded} if there exists a $C \in \R_{\geq 0}$ such that for any $v \in V$, the inequality $\n{f(v)} \leq C \n{v}$ holds. We denote by $\n{f}$ the infimum of such a $C$, and call it {\it the operator norm of $f$}.


\vs
A {\it valuation field} is a field $k$ equipped with a map $\v{\cdot} \colon k \to [0, \infty)$ called a {\it (multiplicative) valuation} satisfying the following:
\bi
\item[(1)] The additive group of $k$ forms a normed Abelian group with respect to $\v{\cdot}$.
\item[(2)] For any $(c_0,c_1) \in k^2$, the equality $\v{c_0 c_1} = \v{c_0} \ \v{c_1}$ holds.
\ei
The reader should be careful not to confound the notations of the valuation $\v{\cdot}$ and the cardinality $\#$.

\vs
For a valuation field $k$, a {\it Banach $k$-vector space structure} on a Banach Abelian group $V$ is a $k$-vector space structure on the underlying set of $V$ compatible with the underlying Abelian group structure of $V$ such that for any $(c,v) \in k \times v$, the equality $\n{cv} = \v{c} \ \n{v}$ holds.

\vs
For a set $I$ and a $V \in \BAb^I$, we denote by $\prod V$ the bounded direct product of $V$, i.e.\ the Banach Abelian group whose underlying set is the set of choice functions $x$ of $V$ such that $\sup_{i \in I} \n{x(i)} < \infty$ and whose norm is the supremum norm, i.e.\ the map $\n{\cdot} \colon \prod V \to \R_{\geq 0}$ defined by
\be
\n{x} \coloneqq \sup_{i \in I} \n{x(i)},
\ee
and by $\bigoplus V$ the completed direct sum of $V$, i.e.\ the closed subspace of $\prod V$ given as
\be
& & \set{x \in \prod V}{\forall \epsilon \in \R_{> 0}[\exists I_0 \in \cP_{< \omega}(I)[\forall i \in I \setminus I_0[\n{x(i)} < \epsilon]]]} \\
& = & \set{x \in \prod V}{\forall \epsilon \in \R_{> 0} \left[ \# \set{i \in I}{\n{x(i)} > \epsilon} < \omega \right]}.
\ee
We do not use $\prod$ and $\bigoplus$ for algebraic direct products or algebraic direct sums.

%% file: Chases_lemma.tex
\section{Chase's Lemma}
\label{Chase's Lemma}

For a set $I$, a $V \in \BAb^I$, an $I_0 \in \cP(I)$, and an $r \in \R_{\geq 0}$, we set
\be
\prod^{I_0,r} V & \coloneqq & \set{x \in \prod V}{\forall i \in I_0[x(i) = 0] \land \forall i \in I[\n{x(i)} \leq r]} \\
\bigoplus^{I_0,r} V & \coloneqq & \set{x \in \bigoplus V}{\forall i \in I \setminus I_0[\n{x(i)} \leq r]}.
\ee
Let $I$ be a set, $J$ a set, $V \in \BAb^I$, $W \in \BAb^J$, and $f$ a non-zero bounded group homomorphism $\prod V \to \bigoplus W$. We have a non-Archimedean analogue of Theorem \ref{Chase's lemma}.

\begin{thm}[non-Archimedean Chase's lemma]
\label{p-adic Chase's lemma}
If $I$ is countable, then for any $r \in \R_{> 0}$, there exists an $(r',I_0,J_0) \in \R_{> \n{f}^{-1} r} \times \cP_{< \omega}(I) \times \cP_{< \omega}(J)$ such that the inclusion
\be
f \left[ \prod^{I_0,r'} V \right] \subset \bigoplus^{J_0,r} W
\ee
holds.
\end{thm}

The proof is completely parallel to that of \cite{Cha62} Theorem 1.2.

\begin{proof}
If $\# I < \omega$, then $(r',I_0,J_0) = (\n{f}^{-1} r + 1,I,\emptyset)$ satisfies the desired property. Therefore, it is reduced to the case $I = \omega$ by the countability of $I$. Assume that the assertion does not hold. For an $n \in \omega$, set $I'_n \coloneqq \omega_{< n}$. We recursively construct sequences $j \in J^{\omega}$, $r' \in \R^{\omega}$, and $x \in (\prod V)^{\omega}$ satisfying the following:
\bi
\item[(1)] For any $n \in \omega$, the inequality $r'(n) > \n{f}^{-1} r$ holds.
\item[(2)] For any $n \in \omega$, the inequality $r'(n) \geq r'(n+1)$ holds.
\item[(3)] For any $n \in \omega$, the relation $x(n) \in \prod^{I'_n,r'(n)} V$ holds.
\item[(4)] For any $n \in \omega$, the inequality $\n{f(x(n))(j(n))} \geq \n{f} r'(n+1)$ holds.
\item[(5)] For any $n \in \omega$ and $h \in \omega_{< n}$, the inequality $\n{f(x(h))(j(n))} \leq r$ holds.
\ei
Set $r'_0 \coloneqq \n{f}^{-1} r + 1$, and $J'_0 \coloneqq \emptyset$. By the hypothesis, the tuple $(r'_0,I'_0,J'_0)$ does not satisfy the desired property. Therefore, there is an $x_0 \in \prod^{I'_0,r'_0} V$ such that $f(x_0) \notin \bigoplus^{J'_0,r} W$. This implies that there is a $j_0 \in J$ such that $\n{f(x_0)(j_0)} > r$. We define $j(0)$ as $j_0$, $r'(0)$ as $r'_0$, $r'(1)$ as $\min \ens{r'_0,\n{f}^{-1} \n{f(x_0)(j_0)}}$, and $x(0)$ as $x_0$. We have constructed $j \upharpoonright \omega_{\leq 0}$, $r' \upharpoonright \omega_{\leq 1}$, and $x \upharpoonright \omega_{\leq 0}$ so that the conditions (1) -- (5) restricted to them hold.

\vs
Suppose that $j \upharpoonright \omega_{\leq n}$, $r' \upharpoonright \omega_{\leq n + 1}$, and $x \upharpoonright \omega_{\leq n}$ have been constructed for an $n \in \omega$ so that the conditions (1) -- (5) restricted to them hold. Set
\be
J'_{n+1} \coloneqq \set{j \in J}{\exists h \in \omega_{\leq n}[\n{f(x(h))(j)} > r]}.
\ee
By the definition of the completed direct sum, we have $\# J'_{n+1} < \omega$. By the hypothesis, the tuple $(r'(n+1),I'_{n+1},J'_{n+1})$ does not satisfy the desired property. Therefore, there is an $x_{n+1} \in \prod^{I'_{n+1},r'(n+1)} V$ such that $f(x_{n+1}) \notin \bigoplus^{J'_{n+1},r} W$. This implies that there is a $j_{n+1} \in J \setminus J'_{n+1}$ such that $\n{f(x_{n+1})(j_{n+1})} > r$. We define $j(n+1)$ as $j_{n+1}$, $r'(n+2)$ as $\min \{r'(n+1),\n{f}^{-1} \n{f(x_{n+1})(j_{n+1})}$, and $x(n+1)$ as $x_{n+1}$. We have constructed $j \upharpoonright \omega_{\leq n+1}$, $r' \upharpoonright \omega_{\leq n+2}$, and $x \upharpoonright \omega_{\leq n+1}$ so that the conditions (1) -- (5) restricted to them hold.

\vs
We have recursively constructed desired $j$, $r'$, and $x$. We define a choice function $s$ of $V$ by
\be
s(i) \coloneqq \sum_{h \in \omega_{\leq i}} x(h)(i) \in V(i).
\ee
By the condition (2), we have $\n{s(i)} \leq r'(0)$ for any $i \in \omega$. This implies $s \in \prod V$. By the conditions (4) and (5), $j$ is injective. Therefore, by the definition of the completed direct sum, there is an $n \in \omega$ such that $\n{f(s)(j(n))} \leq r$. We set
\be
s_{< n} & \coloneqq & \sum_{h \in \omega_{< n}} x(h) \in \prod V \\
s_{> n} & \coloneqq & s - (s_{< n} + x(n)) \in \prod V.
\ee
By the conditions (1) and (5), we have
\be
\n{f(s_{< n})(j(n))} = \n{\sum_{h \in \omega_{< n}} f(x(h))(j(n))} \leq \sup_{h \in \omega_{< n}} \n{f(x(h))(j(n))} \leq r < \n{f} r'(n+1).
\ee
By the conditions (2) and (3), we have $x(h) \in \prod^{\omega_{< h},r'(n+1)} V$ for any $h \in \omega_{> n}$. Therefore, for any $i \in \omega$, we have
\be
\n{s_{> n}(i)} & = & \n{s(i) - (s_{< n}(i) + x(n)(i))} = \n{\left( \sum_{h \in \omega_{\leq i}} x(h)(i) \right) - \left( \sum_{h \in \omega_{< n}} x(h)(i) \right) - x(n)(i)} \\
& = & \n{\left( \sum_{h \in \omega_{\leq i}} x(h)(i) \right) - \left( \sum_{h \in \omega_{\leq n}} x(h)(i) \right)} = \n{\left( \sum_{h \in \omega_{\leq i}} x(h)(i) \right) - \left( \sum_{h \in \omega_{\leq \min \ens{n,i}}} x(h)(i) \right)} \\
& = & 
\left\{
\begin{array}{ll}
0 & (i \leq n) \\
\n{\sum_{h=n+1}^{i} x(h)(i)} & (i > n)
\end{array}
\right.
\leq r'(n+1).
\ee
This implies $\n{s_{> n}} \leq r'(n+1)$, and hence
\be
\n{f(s_{> n})(j(n))} \leq \n{f(s_{> n})} \leq \n{f} r'(n+1).
\ee
By the condition (4), we have
\be
\n{f(x(n))(j(n))} > \n{f} r'(n+1) \geq \max \ens{\n{f(s_{< n})(j(n))}, \n{f(s_{< n})(j(n))}}.
\ee
As a consequence, we obtain
\be
\n{f(s)(j(n))} & = & \n{f(s_{< n} + x(n) + s_{> n})(j(n))} = \n{f(s_{< n})(j(n)) + f(x(n))(j(n)) + f(s_{> n})(j(n))} \\
& = & \n{f(x(n))(j(n))} > \n{f} r'(n+1) > r,
\ee
which contradicts $\n{f(s)(j(n))} \leq r$. Thus, the assertion holds.
\end{proof}

We say that $f$ is {\it $k$-linearisable} for a valuation field $k$ if $f$ is $k$-linear with respect to some Banach $k$-vector space structures on $V(i)$ for each $i \in I$ and $W(j)$ for each $j \in J$. For an $I_0 \in \cP(I)$ and a $J_0 \in \cP(J)$, we denote by
\be
f \upharpoonright^{I_0,J_0} \colon \prod (V \upharpoonright (I \setminus I_0)) \to \bigoplus (W \upharpoonright (J \setminus J_0))
\ee
the composite of the zero extension $\prod (V \upharpoonright (I \setminus I_0)) \hookrightarrow \prod V$, $f$, and the canonical projection $\bigoplus W \twoheadrightarrow \bigoplus (W \upharpoonright (J \setminus J_0))$. As an immediate application of Theorem \ref{p-adic Chase's lemma}, we obtain the following operator norm reduction property:

\begin{crl}
\label{operator norm reduction}
If $I$ is countable and $f$ is $k$-linearisable for a valuation field $k$ such that $\v{k}$ is dense in $\R_{\geq 0}$, then there exists an $(I_0,J_0) \in \cP_{< \omega}(I) \times \cP_{< \omega}(J)$ such that $\n{f \upharpoonright^{I_0,J_0}} < \n{f}$.
\end{crl}

\begin{proof}
By Theorem \ref{p-adic Chase's lemma} applied to $r = 1$, there exists an $(r',I_0,J_0) \in \R_{> \n{f}^{-1}} \times \cP_{< \omega}(I) \times \cP_{< \omega}(J)$ such that the inclusion
\be
f \left[ \prod^{I_0,r'} V \right] \subset \bigoplus^{J_0,1} W
\ee
holds. We show $\n{f \upharpoonright^{I_0,J_0}} \leq (r')^{-1}$. It suffices to show $\n{f \upharpoonright^{I_0,J_0}(x)} \leq (r')^{-1} \n{x}$ for any $x \in \prod (V \upharpoonright (I \setminus I_0))$. If $x = 0$, then we have $\n{f \upharpoonright^{I_0,J_0}(x)} = 0 \leq (r')^{-1} \n{x}$. Suppose $x \neq 0$. Let $\epsilon \in \R_{> 0}$. Since $\v{k}$ is dense in $\R_{\geq 0}$, there exists a $c \in k$ such that $(r' - \epsilon) \n{x}^{-1} < \v{c} \leq r' \n{x}^{-1}$. We have $c \neq 0$ and $r' - \epsilon < \n{cx} \leq r'$. By $x \in \prod (V \upharpoonright (I \setminus I_0))$ and $\n{cx} \leq r'$, the image of $cx$ by the zero extension $\prod (V \upharpoonright (I \setminus I_0)) \hookrightarrow \prod V$ belongs to $\prod^{I_0,r'} V$. Therefore, we have $\n{f \upharpoonright^{I_0,J_0}(cx)} \leq 1$, and hence
\be
\n{f \upharpoonright^{I_0,J_0}(x)} = \v{c}^{-1} \n{f \upharpoonright^{I_0,J_0}(xx)} = \v{c}^{-1} \leq (r' - \epsilon)^{-1} \n{x}
\ee
if $r' - \epsilon > 0$. We obtain
\be
\n{f \upharpoonright^{I_0,J_0}(x)} \leq \inf_{\epsilon \in (0,r')} (r' - \epsilon)^{-1} \n{x} \leq \lim_{\epsilon \to +0} (r' - \epsilon)^{-1} \n{x} = (r')^{-1} \n{x}.
\ee
This implies $\n{f \upharpoonright^{I_0,J_0}} \leq (r')^{-1} < \n{f}$.
\end{proof}

%% file: Edas_Theorem.tex
\section{Chase--Dugas--Zimmermann-Huisgen--Eda Theorem}
\label{Chase--Dugas--Zimmermann-Huisgen--Eda Theorem}

Let $I$ be a set. We denote by $\beta_{\omega} I$ the set of ultrafilters of $I$ closed under countable intersection. Let $V \in \BAb^I$. For a $U_0 \in \cP(\beta_{\omega} I)$, we set
\be
\prod^{U_0,r} V & \coloneqq & \set{x \in \prod V}{\forall \cF \in U_0[\set{i \in I}{x(i) = 0} \in \cF] \land \forall i \in I[\n{x(i)} \leq r]}.
\ee
Let $J$ be a set, $W \in \BAb^J$, and $f$ a non-zero bounded group homomorphism $\prod V \to \bigoplus W$. We have a non-Archimedean analogue of Theorem \ref{Eda's theorem}, which is an extension of Theorem \ref{p-adic Chase's lemma}.

\begin{thm}[non-Archimedean Eda's extension of Dugas--Zimmermann-Huisgen's extension of Chase's lemma]
\label{p-adic Eda's theorem}
For any $r \in \R_{> 0}$, there exists an $(r',U_0,J_0) \in \R_{> \n{f}^{-1} r} \times \cP_{< \omega}(\beta_{\omega} I) \times \cP_{< \omega}(J)$ such that the inclusion
\be
f \left[ \prod^{U_0,r'} V \right] \subset \bigoplus^{J_0,r} W
\ee
holds.
\end{thm}

In order to prove Theorem \ref{p-adic Eda's theorem}, we introduce a set $\cD(I)$ which essentially appears in various studies of Chase's lemma (cf.\ \cite{Eda82} \S 1, \cite{Eda83-1} Theorem 1, \cite{BZ16} Proposition 26, \cite{BR25} Lemma 2.3, and so on), and recall basic properties of $\cD(I)$ well-known to experts for the reader's convenience.

\begin{dfn}
We denote by $\cD(I)$ the set of $\cI \in \cP(\cP(I))$ satisfying the following:
\bi
\item[(1)] The relation $I \notin \cI$ holds.
\item[(2)] For any $I' \in \cP_{< \omega}(I)$, the relation $I' \in \cI$ holds.
\item[(3)] For any $(I',I'') \in \cI^2$, the relation $I' \cup I'' \in \cI$ holds.
\item[(4)] For any $I' \in \cI$ and $I'' \in \cP(I')$, the relation $I'' \in \cI$ holds.
\item[(5)] For any pairwise disjoint $U \in \cP_{\leq \omega}(\cP(I))$, there exists a $U_0 \in \cP_{< \omega}(U)$ such that the relation $\bigcup_{I' \in U \setminus U_0} I' \in \cI$ holds.
\ei
\end{dfn}

\begin{lmm}
\label{closed under countable union}
For any $\cI \in \cD(I)$, $\cI$ is an ideal of $I$ closed under countable union.
\end{lmm}

\begin{proof}
By the condition (2), we have $\emptyset \in \cI$. By the conditions (1), (3), and (4), $\cI$ is an ideal of $I$. We show that $\cI$ is closed under countable union. Let $I' \in \cI^{\omega}$. Set $S \coloneqq \bigcup_{h \in \omega} I'(n)$. It suffices to show $S \in \cI$. By the condition (4), it is reduced to the case where $I'$ is pairwise disjoint. By the condition (5), there exists an $n \in \omega$ such that $\bigcup_{h \in \omega_{\geq n}} I'(h) \in \cI$. By the condition (3), we obtain
\be
S = \bigcup_{i \in \omega_{< n}} I'(h) \cup \bigcup_{h \in \omega_{\geq n}} I'(h) \in \cI.
\ee
Thus, $\cI$ is closed under countable union.
\end{proof}

The following well-foundedness is the fundamental tool to analyse $\cD(I)$:

\begin{lmm}
\label{well-foundedness}
For any $\cI \in \cD(I)$, the following hold:
\bi
\item[(1)] There does not exist a pairwise disjoint $U \in \cP_{= \omega}(\cP(I) \setminus \cI)$.
\item[(2)] For any $\cM \in \cP(\cP(I) \setminus \cI)$, there exists a maximal pairwise disjoint $U \in \cP(\cM)$.
\ei
\end{lmm}

\begin{proof}
The assertion (2) immediately follows from the assertion (1). Assume that there exists a pairwise disjoint $U \in \cP_{= \omega}(\cP(I) \setminus \cI)$. By the condition (5), there exists a $U_0 \in \cP_{< \omega}(U)$ such that $\bigcup_{I' \in U \setminus U_0} I' \in \cI$. By $\# U = \omega > \# U_0$, $U \setminus U_0$ has an element $I''$. We have $I'' \subset \bigcup_{I' \in U \setminus U_0} I'$, and hence $I'' \in \cI$ by the condition (4), which contradicts $I'' \in U \subset \cP(I) \setminus \cI$.
\end{proof}

\begin{lmm}
\label{countably complete measure}
For any $\cI \in \cD(I)$, there exists an $M \in \cP(I) \setminus \cI$ such that for any $I' \in \cP(M)$, either $I' \in \cI$ or $M \setminus I' \in \cI$ holds.
\end{lmm}

The proof is essentially identical to that of \cite{BR25} Lemma 2.3, which is reduced to arguments in \cite{DZH82} Theorem 2 and \cite{BZ16} Proposition 26.

\begin{proof}
Assume the non-existence of such an $M$. We construct an $I' \in (\cP(I) \setminus \cI)^{\omega}$ in a recursive way so that for any $n \in \omega$, the relation $I \setminus \bigcup_{h \in \omega_{< n}} I'(h) \notin \cI$ holds.

\vs
By the condition (1), we have $I \in \cP(I) \setminus \cI$. By $I \in \cP(I) \setminus \cI$ and the hypothesis, there exists an $I'_0 \in \cP(I')$ such that neither $I'_0 \in \cI$ nor $I \setminus I'_0 \in \cI$ holds. We define $I'(0) \coloneqq I'_0$, and have constructed $I' \upharpoonright \omega_{\geq 0}$ with
\be
I \setminus \bigcup_{h \in \omega_{\leq 0}} I'(h) = I \setminus I'_0 \notin \cI.
\ee
Suppose that we have constructed $I' \upharpoonright \omega_{\geq n} \in (\cP(I) \setminus \cI)^{\omega_{\geq 0}}$ for an $n \in \omega$ so that for any $n' \in \omega_{\leq n}$, the relation $I \setminus \bigcup_{h \in \omega_{< n'}} I'(h) \notin \cI$ holds. Set $S \coloneqq I \setminus \bigcup_{h \in \omega_{< n}} I'(h)$. By $S \in \cP(I) \setminus \cI$ and the hypothesis, there exists an $I'_{n+1} \in \cP(S)$ such that neither $I'_{n+1} \in \cI$ nor $S \setminus I'_{n+1} \in \cI$ holds. We define $I'(n+1) \coloneqq I'_{n+1}$, and have constructed $I' \upharpoonright \omega_{\geq n+1}$ with
\be
I \setminus \bigcup_{h \in \omega_{\leq n+1}} I'(h) = S \setminus I'_{n+1} \notin \cI.
\ee
We have recursively constructed a desired $I' \in (\cP(I) \setminus \cI)^{\omega}$. By the construction, $I'$ is pairwise disjoint. This contradicts Lemma \ref{well-foundedness} (1).
\end{proof}

\begin{lmm}
\label{finite maximal pairwise disjoint system}
For any $\cI \in \cD(I)$, there exists a pairwise disjoint $U_0 \in \cP_{< \omega}(\cP(I) \setminus \cI)$ satisfying the following:
\bi
\item[(1)] For any $M \in U_0$ and any $I' \in \cP(M)$, either $I' \in \cI$ or $M \setminus I' \in \cI$ holds.
\item[(2)] The inclusion $\cP(I \setminus \bigcup_{M \in U_0} M) \subset \cI$ holds.
\ei
\end{lmm}

The proof is parallel to that of \cite{Eda82} Lemma 2 on Boolean power.

\begin{proof}
We denote by $\cM$ the set of an $M \in \cP(I) \setminus \cI$ such that for any $I' \in \cP(M)$, either $I' \in \cI$ or $M \setminus I' \in \cI$ holds. By Lemma \ref{well-foundedness} (2), there exists a maximal pairwise disjoint $U_0 \in \cP(\cM)$. By Lemma \ref{well-foundedness} (1), $U_0$ is finite. Set $S \coloneqq \bigcup_{I' \in U_0} I'$.

\vs
It suffices to show $\cP(I \setminus S) \subset \cI$. Let $I' \in \cP(I \setminus S)$. We show $I' \in \cI$. Assume $I' \notin \cI$. We construct a pairwise disjoint $I'' \in (\cP(I \setminus S) \setminus (\cI \cup \cM))^{\omega}$ in a recursive way so that for any $n \in \omega$, the relation $I' \setminus \bigcup_{h \in \omega_{\leq n}} I''(h) \notin \cI \cup \cM$ holds.

\vs
By $I' \subset I \setminus S$ and the maximality of $U_0$, we have $I' \notin \cM$. By $I' \notin \cI$ and $I' \notin \cM$, there exists an $I''_0 \in \cP(I')$ such that neither $I''_0 \in \cI$ nor $I' \setminus I''_0 \in \cI$ holds. By $I' \setminus I''_0 \subset I \setminus S$ and the maximality of $U$, we have $I' \setminus I''_0 \notin \cM$. We define $I''(0) \coloneqq I''_0$, and have constructed $I'' \upharpoonright \omega_{\leq 0}$ with
\be
I' \setminus \bigcup_{h \in \omega_{\leq 0}} I''(h) = I' \setminus I''_0 \notin \cI \cup \cM.
\ee
Suppose that we have constructed $I'' \upharpoonright \omega_{\leq n}$ for an $n \in \omega$ so that for any $n' \in \omega_{\leq n}$, the relation $I' \setminus \bigcup_{h \in \omega_{\leq n'}} I''(h) \notin \cI \cup \cM$ holds. Set $S' \coloneqq I' \setminus \bigcup_{h \in \omega_{\leq n}} I''(h)$. By $S' \notin \cI$ and $S' \notin \cM$, there exists an $I''_{n+1} \in \cP(S')$ such that neither $I''_{n+1} \in \cI$ nor $S' \setminus I''_{n+1} \in \cI$ holds. By $S' \setminus I''_{n+1} \subset I' \subset I \setminus S$ and the maximality of $U$, we have $S' \setminus I''_{n+1} \notin \cM$. We define $I''(n+1) \coloneqq I''_{n+1}$, and have constructed $I'' \upharpoonright \omega_{\leq n+1}$ with
\be
I' \setminus \bigcup_{h \in \omega_{\leq n+1}} I''(h) = S' \setminus I''_{n+1} \notin \cI \cup \cM.
\ee
We have recursively constructed a pairwise disjoint $I'' \in (\cP(I \setminus S) \setminus (\cI \cup \cM))^{\omega}$. This contradicts Lemma \ref{well-foundedness} (1). This implies $I' \in \cI$. We conclude $\cP(I \setminus S) \subset \cI$.
\end{proof}

For a set $I'$, an $I'' \in \cP(I')$, a $V' \in \BAb^{I''}$, an $I''_0 \in \cP(I'')$, and an $r' \in \R_{> 0}$, we set
\be
V' \upharpoonright_{I''}^{I''_0,r'} \coloneqq \prod^{I''_0 \cup (I' \setminus I''),r'} V'.
\ee
For an $r \in \R_{> 0}$, we denote by $\cI_{f,r}$ the set of subsets $I' \subset I$ such that there exists an $(r',I'_0,J_0) \in \R_{> \n{f}^{-1} r} \times \cP_{< \omega}(I) \times \cP_{< \omega}(J)$ such that the inclusion
\be
f \left[ V \upharpoonright_{I'}^{I'_0,r'} \right] \subset \bigoplus^{J_0,r} W
\ee
holds.

\begin{lmm}
\label{cI is cD}
For any $r \in \R_{> 0}$, if $I \notin \cI_{f,r}$, then the relation $\cI_{f,r} \in \cD(I)$ holds.
\end{lmm}

\begin{proof}
We show the conditions (1) -- (5) for $\cI_{f,r} \in \cD(I)$. The assumption ensures (1). Concerning (2), for any $I' \in \cP_{< \omega}(I)$, by $V \upharpoonright_{I'}^{I',\n{f}^{-1} r + 1} = \ens{0}$, we have $I' \in \cI_{f,r}$. Concerning (4), for any $I' \in \cI_{f,r}$ and any $I'' \in \cP(I')$, we have $\prod^{(I'_0 \cap I'') \cup (I \setminus I''),r'} V \subset \prod^{I'_0 \cup (I \setminus I'),r'} V$ for any $I'_0 \in \cP_{< \omega}(I')$ and hence $I'' \in \cI_{f,r}$.

\vs
We show (3). Let $I' \in \cI_{f,r}^2$. By the definition of $\cI_{f,r}$, for each $h \in \ens{0,1}$, there exists an $(r'_h,I'_h,J_h) \in \R_{> \n{f}^{-1} r} \times \cP_{< \omega}(I) \times \cP_{< \omega}(J)$ such that the inclusion
\be
f \left[ V \upharpoonright_{I'(h)}^{I'_h,r'_h} \right] \subset \bigoplus^{J_h,r} W
\ee
holds. Set $r' \coloneqq \min \ens{r'_0,r'_1}$, $I'' \coloneqq I'_0 \cup I'_1$, and $J' \coloneqq J_0 \cup J_1$. Then we have
\be
f \left[ V \upharpoonright_{I'(0) \cup I'(1)}^{I'',r''} \right] \subset f \left[ \sum_{h \in \ens{0,1}} V \upharpoonright_{I'(h)}^{I'_h,r'_h} \right] \subset \sum_{h \in \ens{0,1}} f \left[ V \upharpoonright_{I'(h)}^{I'_h,r'_h} \right] \subset \sum_{h \in \ens{0,1}} \bigoplus^{J_h,r} W \subset \bigoplus^{J',r} W,
\ee
and hence $I'(0) \cup I'(1) \in \cI_{f,r}$. We have shown (3).

\vs
We show (5). Let $U \in \cP_{\leq \omega}(\cP(I))$, and suppose that $U$ is pairwise disjoint. It suffices to show that there exists a $U_0 \in \cP_{< \omega}(U)$ such that $\bigcup_{I' \in U \setminus U_0} I' \in \cI_{f,r}$. Set $S \coloneqq \bigcup_{I' \in U} I'$. We define a $V' \in \BAb^U$ by
\be
V'(I') \coloneqq \prod (V \upharpoonright I') \in \BAb.
\ee
Through the natural identification of $\prod V'$ and $\set{v \in \prod V}{\forall i \in I \setminus S[v(i) = 0]}$, we regard $\prod V'$ as a closed subgroup of $\prod V$. Set $f' \coloneqq f \upharpoonright \prod V'$. If $f'$ is zero, then $U_0 = \emptyset$ satisfies the desired condition, i.e.\ $S \in \cI_{f,r}$, because $(r',I'_0,J_0) = (\n{f}^{-1} r + 1,\emptyset,\emptyset)$ satisfies the desired condition. Therefore, it is reduce to the case where $f'$ is not zero.

\vs
By Theorem \ref{p-adic Chase's lemma} applied to $f'$, there exists an $(r'_0,U_0,J'_0) \in \R_{> \n{f'}^{-1} r} \times \cP_{< \omega}(U) \times \cP_{< \omega}(J)$ such that the inclusion
\be
f' \left[ V' \upharpoonright_{U}^{U_0,r'_0} \right] \subset \bigoplus^{J'_0,r} W
\ee
holds. By $\n{f'} \leq \n{f}$, we have $r'_0 > \n{f}^{-1} r$. Set $I' \coloneqq \bigcup_{I'' \in U \setminus U_0} I''$. We have
\be
f \left[ V \upharpoonright_{I'}^{\emptyset,r'_0} \right] = f' \left[ V' \upharpoonright_{U}^{U_0,r'_0} \right] \subset \bigoplus^{J'_0,r} W,
\ee
and hence $I' \in \cI_{f,r}$. We have shown (5). Thus, we obtain $\cI_{f,r} \in \cD(I)$.
\end{proof}

We go back to Theorem \ref{p-adic Eda's theorem}. For any $i \in I$, the principal ultrafilter of $I$ associated to $i$ is closed under intersection, and hence is an element of $\beta_{\omega} I$. We denote by $\iota_I$ the canonical embedding $I \hookrightarrow \beta_{\omega} I$.

\begin{proof}[Proof of Theorem \ref{p-adic Eda's theorem}]
If $I \in \cI_{f,r}$, then the assertion holds because every principal ultrafilter is $\omega_1$-complete. Therefore, it is reduced to the case $I \notin \cI_{f,r}$, i.e.\ $\cI_{f,r} \in \cD(I)$ by Lemma \ref{cI is cD}. By Lemma \ref{finite maximal pairwise disjoint system}, there exists a pairwise disjoint $U \in \cP_{< \omega}(\cP(I) \setminus \cI_{f,r})$ satisfying the following:
\bi
\item[(1)] For any $M \in U$ and any $I' \in \cP(M)$, either $I' \in \cI_{f,r}$ or $M \setminus I' \in \cI_{f,r}$ holds.
\item[(2)] The inclusion $\cP(I \setminus \bigcup_{M \in U} M) \subset \cI_{f,r}$ holds.
\ei
We define an $\cF \in \cP(\cP(I))^U$ by
\be
\cF(M) \coloneqq \set{I' \in \cP(I)}{M \setminus I' \in \cI_{f,r}}.
\ee
Let $M \in U$. By Lemma \ref{closed under countable union}, $\cI_{f,r}$ is an ideal of $I$ closed under countable union, and hence $\cF(M)$ is a filter of $I$ closed under countable intersection. By the condition (1) of $U$, $\cF(M)$ is an ultrafilter. This implies $\cF \in (\beta_{\omega} I)^U$.

\vs
Set $S \coloneqq \bigcup_{M \in U} M$. By the condition (2) of $U$, we have $I \setminus S \in \cI_{f,r}$, and hence there exists an $(r',I_0,J_0) \in \R_{> \n{f}^{-1} r} \times \cP_{< \omega}(I \setminus S) \times \cP_{< \omega}(J)$ such that the inclusion
\be
f \left[ V \upharpoonright_{I \setminus S}^{I_0,r'} V \right] \subset \bigoplus^{J_0,r} W
\ee
holds. Set $U_0 \coloneqq \cF[U]$. We show that for any $r' \in \R_{> \n{f}^{-1} r}$ and any $x \in \prod^{U_0,r'} V$, the relation $\set{i \in I}{x(i) \neq 0} \in \cI_{f,r}$ holds. Set $I' \coloneqq \set{i \in I}{x(i) \neq 0}$. For any $M \in U$, we have $I \setminus I' \in \cF(M)$ by $x \in \prod^{U_0,r'} V$, and hence $M \cap I' = M \setminus (I \setminus I') \in \cI_{f,r}$. By the condition (3) of $\cI_{f,r} \in \cD(I)$, we obtain
\be
I' = \left( I' \setminus S \right) \cup \bigcup_{M \in U} (M \cap I') \in \cI_{f,r}.
\ee
For an $I_0 \in \cP(I)$, and an $r' \in \R_{> 0}$, we set
\be
\prod^{U_0,I_0,r'} V \coloneqq \prod^{U_0 \cup \iota_I[I_0],r'} V.
\ee
By $U_0 \cup \iota_I[I_0] \in \cP_{< \omega}(\beta_{\omega} I)$ for any $I_0 \in \cP_{< \omega}(I)$, it suffices to show that there exists an $(r',I_0,J_0) \in \R_{> \n{f}^{-1} r} \times \cP_{< \omega}(I) \times \cP_{< \omega}(J)$ such that the inclusion
\be
f \left[ \prod^{U_0,I_0,r'} V\right] \subset \bigoplus^{J_0,r} W
\ee
holds. Assume the non-existence of such an $(r',I_0,J_0)$. We construct a sequence $x \in (\prod V)^{\omega}$ in a recursive way so that $x$ and the sequences $r' \in \R_{> \n{f}^{-1} r}^{\omega}$, $I' \in \cP_{< \omega}(I)^{\omega}$, and $J' \in \cP_{< \omega}(J)^{\omega}$ defined by
\be
r'(n) & \coloneqq & \n{f}^{-1} r + 2^{-n} \\
I'(n) & \coloneqq & \set{i \in I}{\n{x(n)(i)} > \n{f}^{-1} r} \\
J'(n) & \coloneqq & \set{j \in J}{\n{f(x(n))(j)} > r}
\ee
satisfy the following conditions:
\bi
\item[(1)] For any $n \in \omega$, the inequalities $\n{x(n)} \leq \n{f}^{-1} r + 2^{-n}$ and $\n{f(x(n))} > r$ hold.
\item[(2)] The sequence $I'$ is pairwise disjoint, and satisfies $I'(n) = \set{i \in I}{x(n)(i) \neq 0}$ for any $n \in \omega$.
\item[(3)] The sequence $J'$ satisfies $J'(n) \setminus \bigcup_{h \in \omega_{< n}} J'(h) \neq \emptyset$ for any $n \in \omega$.
\ei
Suppose that we have constructed $x \upharpoonright \omega_{< n}$ for an $n \in \omega$ so that the conditions (1) -- (3) restricted to it hold. Set $S_n \coloneqq \bigcup_{h \in \omega_{< n}} I'(h)$ and $T_n \coloneqq \bigcup_{h \in \omega_{< n}} J'(h)$. By the hypothesis, there exists an $x_n \in \prod^{U_0,S_n,r'(n)} V$ such that $f(x_n) \notin \bigoplus^{T_n,r} W$. Replacing $x_n(i)$ by $0 \in V(i)$ for each $i \in I$ with $\n{x_n(i)} \leq \n{f}^{-1} r$, we may assume $\set{i \in I}{0 < \n{f(x(n))} \leq \n{f}^{-1} r} = \emptyset$. We define $x(n) \coloneqq x_n$, and have constructed $x \upharpoonright \omega_{< n + 1}$ so that the conditions (1) -- (3) restricted to them hold.

\vs
We have recursively constructed a desired $x$. Set $I'' \coloneqq \bigcup_{h \in \omega} I'(h)$. Since $\cI_{f,r}$ is closed under countable union, we have $I'' \in \cI_{f,r}$. Therefore, there exists an $(r'',I''_0,J''_0) \in \R_{> \n{f}^{-1} r} \times \cP_{< \omega}(I'') \times \cP_{< \omega}(J)$ such that the inclusion
\be
f \left[ V \upharpoonright_{I''}^{I''_0,r''} V \right] \subset \bigoplus^{J''_0,r} W
\ee
holds. By $r'' > \n{f}^{-1} r$, $\# I''_0 < \omega$, and $\# J''_0 < \omega$, there exists an $n \in \omega$ such that $r'' \geq \n{f}^{-1} r + 2^{-n}$, $I''_0 \cap I'(n) = \emptyset$, and $J'(n) \setminus J''_0 \neq \emptyset$. We have $x(n) \in V \upharpoonright_{I''}^{I''_0,r''}$ and hence $f(x(n)) \in \bigoplus^{J''_0,r} W$, which contradicts
\be
\set{j \in J}{\n{f(x(n))(j)} > r} \setminus J''_0 = J'(n) \setminus J''_0 \neq \emptyset.
\ee
Thus, the hypothesis is false.
\end{proof}

An ultrafilter $\cF$ is said to be {\it $\lambda$-complete} for a cardinal number $\lambda$ if $\bigcap_{U \in F} U \in \cF$ for any $F \in \cP_{< \lambda}(\cF) \setminus \ens{\emptyset}$. A set $X$ is said to be {\it $\lambda$-measurable} for a cardinal number $\lambda$ if $X$ is uncountable and admits a $\lambda$-complete non-principal ultrafilter. A cardinal number $\mu$ is said to be {\it measurable} if $\mu$ is $\mu$-measurable. We denote by $\omega_1$ the least uncountable ordinal $\aleph_1$. We recall that a set $X$ is $\omega_1$-measurable if and only if there exists a measurable cardinal smaller than or equal to $\# X$. Since the existence of a measurable cardinal is unprovable under $\textsf{ZFC}$ as long as $\textsf{ZFC}$ is consistent, so is the existence of an $\omega_1$-measurable set. As a consequence of Theorem \ref{p-adic Eda's theorem}, we obtain a non-Archimedean analogue of Theorem \ref{Dugas--Zimmermann-Huisgen's theorem}:

\begin{crl}[non-Archimedean Dugas--Zimmermann-Huisgen's extension of Chase's lemma]
\label{p-adic Dugas--Zimmermann-Huisgen's theorem}
If $I$ is not $\omega_1$-measurable, then for any $r \in \R_{> 0}$, there exists an $(r',I_0,J_0) \in \R_{> \n{f}^{-1} r} \times \cP_{< \omega}(I) \times \cP_{< \omega}(J)$ such that the inclusion
\be
f \left[ \prod^{I_0,r'} V \right] \subset \bigoplus^{J_0,r} W
\ee
holds.
\end{crl}

We note that Corollary \ref{p-adic Dugas--Zimmermann-Huisgen's theorem} immediately follows from Lemma \ref{countably complete measure} and Lemma \ref{cI is cD}, because they imply $I \in \cI_{f,r}$. We give an alternative proof in order to make it clear that Corollary \ref{p-adic Dugas--Zimmermann-Huisgen's theorem} is a special case of Theorem \ref{p-adic Eda's theorem}.

\begin{proof}
By Theorem \ref{p-adic Eda's theorem}, there exists an $(r',U_0,J_0) \in \R_{> \n{f}^{-1} r} \times \cP_{< \omega}(\beta_{\omega} I) \times \cP_{< \omega}(J)$ such that the inclusion
\be
f \left[ \prod^{U_0,r'} V \right] \subset \bigoplus^{J_0,r} W
\ee
holds. Since $I$ is not $\omega_1$-measurable, we have $\beta_{\omega} I = \iota_I[I]$. Therefore, there exists an $I_0 \in \cP_{< \omega}(I)$ such that $U_0 = \iota_I[I_0]$. We have
\be
\prod^{U_0,r'} V & = & \set{x \in V}{\forall \cF \in U_0[\set{i \in I}{x(i) = 0} \in \cF]} = \set{x \in V}{\forall i_0 \in I_0[\set{i \in I}{x(i) = 0} \in \iota_I(i_0)]} \\
& = & \set{x \in V}{\forall i_0 \in I_0[x(i_0) = 0]} = \prod^{I_0,r} V,
\ee
and hence $(r',I_0,J_0)$ satisfies the desired condition.
\end{proof}

As an extension of Corollary \ref{operator norm reduction}, we obtain the following operator norm reduction property:

\begin{crl}
If $I$ is not $\omega_1$-measurable and $f$ is $k$-linearisable for a valuation field $k$ such that $\v{k}$ is dense in $\R_{\geq 0}$, then there exists an $(I_0,J_0) \in \cP_{< \omega}(I) \times \cP_{< \omega}(J)$ such that $\n{f \upharpoonright^{I_0,J_0}} < \n{f}$.
\end{crl}

\begin{proof}
The assertion immediately follows from Corollary \ref{p-adic Dugas--Zimmermann-Huisgen's theorem} by an argument completely parallel to the proof of Corollary \ref{operator norm reduction}.
\end{proof}

%% file: References.tex
\vspace{0.3in}
\addcontentsline{toc}{section}{Acknowledgements}
\noindent {\Large \bf Acknowledgements}
\vspace{0.2in}

\noindent
I thank K.\ Eda for introducing to me almost everything in this paper: classical results on Specker phenomenon, {\L}o\'s--Eda theorem for $\Z^I$ and $\ell^{\infty}(I,k)$, Chase's lemma, Eda's extension of Dugas--Zimmermann-Huisgen's extension of Chase's lemma, and expectation that a non-Archimedean counterpart of Chase's lemma should exist. I thank K.\ Ishizuka for informing me of several preceding studies and references on reflexivity in the non-Archimedean setting. I thank Y.\ Isono for recalling me of basic facts on reflexivity in the Archimedean setting. I thank all people who helped me to learn mathematics and programming. I also thank my family.

%